\documentclass[english]{article}
\usepackage[T1]{fontenc}
\usepackage[latin9]{inputenc}
\usepackage{amsmath}
\usepackage{amsthm}
\usepackage{amssymb}

\makeatletter
\theoremstyle{plain}
\newtheorem{thm}{\protect\theoremname}
\theoremstyle{plain}
\newtheorem{lem}[thm]{\protect\lemmaname}
\theoremstyle{plain}
\newtheorem{cor}[thm]{\protect\corollaryname}
\theoremstyle{remark}
\newtheorem{claim}[thm]{\protect\claimname}

\usepackage{todonotes}
\usepackage{microtype}
\usepackage{fullpage}

\makeatother

\usepackage{babel}
\providecommand{\claimname}{Claim}
\providecommand{\corollaryname}{Corollary}
\providecommand{\lemmaname}{Lemma}
\providecommand{\theoremname}{Theorem}

\begin{document}
\global\long\def\sharpETH{\mathsf{\#ETH}}%
\global\long\def\sharpWone{\mathsf{\#W[1]}}%
\global\long\def\sharpP{\mathsf{\#P}}%
\global\long\def\PerfMatch{\#\mathrm{PerfMatch}}%
\global\long\def\Sig{\mathrm{Sig}}%
\global\long\def\hadw{\mathrm{hadw}}%
\global\long\def\Ext{\mathrm{Ext}}%
\global\long\def\sgn{\mathrm{sgn}}%

\title{Determinants from homomorphisms}
\author{Radu Curticapean\thanks{IT University of Copenhagen, BARC Copenhagen. Supported by Villum Foundation grant 16582. racu@itu.dk}}
\maketitle
\begin{abstract}
We give a new combinatorial explanation for well-known relations between
determinants and traces of matrix powers. Such relations can be used
to obtain polynomial-time and poly-logarithmic space algorithms for
the determinant. Our new explanation avoids linear-algebraic arguments
and instead exploits a classical connection between subgraph and homomorphism
counts.
\end{abstract}

\section{Introduction}

The $n\times n$ determinant is (up to scaling) the unique function
from $n\times n$ matrices to scalars that is multilinear and alternating
in the columns. It admits the \emph{Leibniz formula}
\begin{equation}
\det(A)=\sum_{\pi\in S_{n}}\sgn(\pi)\prod_{i=1}^{n}a_{i,\pi(i)},\label{eq: leibniz}
\end{equation}
where $S_{n}$ is the set of permutations of $\{1,\ldots,n\}$ and
$\sgn$ is the permutation sign. Writing $\sigma(\pi)$ for the number
of cycles in $\pi$, the permutation sign can be expressed as $\sgn(\pi)=(-1)^{n+\sigma(\pi)}$. 

When presented with only the right-hand side of (\ref{eq: leibniz}),
unaware of the connection to the determinant, one would likely struggle
to evaluate this sum of $n!$ terms in polynomial time. For comparison,
it is $\sharpP$-hard to compute the similarly defined \emph{permanent}~\cite{Valiant1979a},
which is obtained by omitting the sign factors from (\ref{eq: leibniz}).

Yet, determinants can be evaluated efficiently, e.g., via Gaussian
elimination in $O(n^{3})$ field operations, including divisions.
Asymptotically optimal algorithms achieve $O(n^{\omega})$ operations,
where $\omega<3$ is the exponent of matrix multiplication~\cite[Exercise~28.2-3]{DBLP:books/daglib/0023376}.
Note that (\ref{eq: leibniz}) is defined over any ring containing
the entries of $A$; it is also known that $\det(A)$ can be computed
with a polynomial number of ring operations, i.e., excluding divisions~\cite{DBLP:conf/soda/MahajanV97,DBLP:journals/ipl/Berkowitz84,10.2307/2235845}.

\subsection*{Determinants from matrix powers}

It is a classical result in linear algebra that $\det(A)$ can be
computed from the matrix traces $\mathrm{tr}(A^{k})$ for $1\leq k\leq n$.
Assume that $A$ is defined over an algebraically closed field $\mathbb{F}$
of characteristic $0$, such that $A$ has eigenvalues $\lambda_{1},\ldots,\lambda_{n}\in\mathbb{F}$.
The idea is to express $\det(A)$ and $\mathrm{tr}(A^{k})$ for $1\leq k\leq n$
as particular polynomials in the eigenvalues $\lambda_{1},\ldots,\lambda_{n}$
and then relate these polynomials.
\begin{itemize}
\item The determinant can be expressed as $\det(A)=\lambda_{1}\ldots\lambda_{n}$;
this is the $n$-th elementary symmetric polynomial in the eigenvalues.
Generally, the $k$-th elementary symmetric polynomial $e_{k}(x_{1},\ldots,x_{n})$
in $n$ variables is the sum of monomials $\sum_{S}\prod_{i\in S}x_{i}$,
where $S$ ranges over all $k$-subsets $S\subseteq\{1,\ldots,n\}$.
\item The matrix trace satisfies $\mathrm{tr}(A)=\lambda_{1}+\ldots+\lambda_{n}$,
and more generally, $\mathrm{tr}(A^{k})=\lambda_{1}^{k}+\ldots+\lambda_{n}^{k}$;
this is the $k$-th power-sum polynomial in the eigenvalues. Generally,
the $k$-th power sum polynomial $p_{k}(x_{1},\ldots,x_{n})$ in $n$
variables is $x_{1}^{k}+\ldots+x_{n}^{k}$.
\end{itemize}
The Girard--Newton identities then relate the power-sum polynomials
with the elementary symmetric polynomials. That is,
\[
ke_{k}(x_{1},\ldots,x_{n})=\sum_{i=1}^{k}(-1)^{i-1}e_{k-i}(x_{1},\ldots,x_{n})p_{i}(x_{1},\ldots,x_{n}).
\]
A recursive application of these identities yields $\det(A)=e_{n}(\lambda_{1},\ldots,\lambda_{n})$
from the traces of matrix powers $\mathrm{tr}(A^{k})=p_{k}(\lambda_{1},\ldots,\lambda_{n})$
for $1\leq k\leq n$. Csanky's algorithm~\cite[Chapter~31]{Kozen1992}
implements this approach with an arithmetic circuit of bounded fan-in,
$O(\log^{2}n)$ depth, and polynomial size. In other words, it shows
that determinants can be computed with $O(\log^{2}n)$ operations
on a polynomial number of parallel processors.

\subsection*{Our results}

Some notation first: A partition $\lambda$ of a number $n\in\mathbb{N}$
is a multi-set of numbers summing to $n$. We write $\lambda\vdash n$
to indicate that $\lambda$ is a partition of $n$ and write $|\lambda|$
for its number of parts. For $1\leq\ell\leq n$, we write $s_{\ell}(\lambda)\in\mathbb{N}$
for the number of occurrences of $\ell$ in $\lambda$.

The main result of this paper is a novel and self-contained derivation
of the known formula
\begin{equation}
\mathrm{det}(A)=(-1)^{n}\sum_{\lambda\vdash n}(-1)^{|\lambda|}\prod_{\ell=1}^{n}\frac{\mathrm{tr}(A^{\ell})^{s_{\ell}(\lambda)}}{s_{\ell}(\lambda)!\cdot\ell^{s_{\ell}(\lambda)}}.\label{eq: det-partsum}
\end{equation}
By expanding the Girard--Newton identities, the previous subsection
essentially implies a proof of this formula. Our new proof bypasses
notions like eigenvalues, symmetric polynomials, and the Girard--Newton
identities, and instead relies on graph-theoretic ideas. The proof
is contained in Section~\ref{sec:Proof-of}.

In Section~\ref{sec:Algorithmic-applications}, we show how this
formula can be used to obtain a polynomial-time algorithm for the
determinant, which can also be parallelized.

\section{\label{sec:Proof-of}Proof of (\ref{eq: det-partsum})}

In the following, let $A=(a_{i,j})_{i,j\in[n]}$ be a matrix. We will
study the determinant of $A$ using graph-theoretic language. The
graphs in this paper are \emph{directed} and may feature \emph{self-loops},
and some graphs may feature \emph{parallel edges} between the same
vertices. We write $V(G)$ and $E(G)$ for the vertices and edges
of a graph $G$.

\subsection{Determinants are sums of cycle covers}

The matrix $A$ induces a weighted complete directed graph (with self-loops)
on vertex set $V(A)=[n]$. Abusing notation, we also write $A$ for
this weighted graph. In this view, permutations correspond to \emph{cycle
covers}, which are edge-sets $C\subseteq E(A)$ inducing vertex-disjoint
cycles that cover all vertices of $A$. More generally, a $k$-partial
cycle cover for $0\leq k\leq n$ is a vertex-disjoint collection of
cycles with $k$ edges in total. Its \emph{format} is the partition
$\lambda\vdash k$ induced by the multi-set of cycle lengths. We write
$\mathcal{C}(n,k)$ for the set of all $k$-partial cycle covers of
the complete directed graph on vertex set $[n]$ and define
\begin{equation}
\sgn(C)=(-1)^{|C|+\sigma(\pi)}\label{eq: def-partial-sign}
\end{equation}
analogously to the permutation sign. Our proofs rely on the $k$-partial
determinant

\[
\mathrm{det}_{k}(A)=\sum_{\substack{S\subseteq[n]\text{ of size }k}
}\det(A[S]),
\]
where $A[S]$ is the square sub-matrix of $A$ defined by restricting
to the rows and columns contained in $S$. From the Leibniz formula
(\ref{eq: leibniz}), it follows that
\begin{align}
\mathrm{det}_{k}(A) & =\sum_{C\in\mathcal{C}(n,k)}\mathrm{sgn}(C)\prod_{uv\in C}a_{u,v}.\label{eq: det-cc}
\end{align}
Given $\lambda\vdash k$, let $C_{\lambda}$ be the cycle cover with
one cycle of length $s$ for each part $s$ in $\lambda$. We can
regroup terms in (\ref{eq: det-cc}) to obtain
\begin{equation}
\mathrm{det}_{k}(A)=\sum_{\lambda\vdash k}\mathrm{sgn}(C_{\lambda})\underbrace{\sum_{\substack{C\in\mathcal{C}(n,k)\,\mathrm{of}\\
\mathrm{format}\ \lambda
}
}\prod_{\ uv\in C}a_{u,v}}_{=:\mathrm{sub}(C_{\lambda}\to A)}.\label{eq: cycles-by-format}
\end{equation}
Note that $\mathrm{sub}(C_{\lambda}\to A)$ counts subgraphs isomorphic
to $C_{\lambda}$ in $A$, weighted by the product of the edge-weights
in the subgraph.

\subsection{Relating subgraphs, embeddings and homomorphisms}

Let $L$ be a graph, possibly containing parallel edges. The weighted
\emph{homomorphism} and \emph{embedding} counts from $L$ into $A$
are defined as 
\begin{align}
\hom(L\to A) & =\sum_{f:V(L)\to[n]}\prod_{\substack{e\in E(L)\\
\mathrm{with}\,e=uv
}
}a_{f(u),f(v)},\label{eq: hom-def}\\
\mathrm{emb}(L\to A) & =\sum_{\substack{f:V(L)\to[n]\\
\mathrm{injective}
}
}\prod_{\substack{e\in E(L)\\
\mathrm{with}\,e=uv
}
}a_{f(u),f(v)}.
\end{align}

For example, if $C_{\ell}$ denotes the $\ell$-cycle, then $\hom(C_{\ell}\to A)=\mathrm{tr}(A^{\ell})$.
Moreover, we have 
\begin{equation}
\hom(C_{\lambda}\to A)=\prod_{\ell=1}^{n}\mathrm{tr}(A^{\ell})^{s_{\ell}(\lambda)},\label{eq: hom-cc}
\end{equation}
since homomorphisms from a disjoint union of graphs can be chosen
independently for the individual components. Recall that $s_{\ell}(\lambda)$
counts the occurrences of part $\ell$ in $\lambda$.

Given a graph $P$ without parallel edges, an \emph{automorphism}
of $P$ is an isomorphism into itself. We write $\mathrm{aut}(P)$
for the number of automorphisms of $P$. For example,
\begin{equation}
\mathrm{aut}(C_{\lambda})=\prod_{\ell=1}^{n}s_{\ell}(\lambda)!\cdot\ell^{s_{\ell}(\lambda)}.\label{eq: aut-cc}
\end{equation}

Subgraph counts from cycle covers, as defined in (\ref{eq: cycles-by-format}),
are related to embedding counts via automorphisms:
\begin{equation}
\mathrm{sub}(C_{\lambda}\to A)=\frac{\mathrm{emb}(C_{\lambda}\to A)}{\mathrm{aut}(C_{\lambda})}.\label{eq: sub-emb-aut}
\end{equation}

Embedding counts and homomorphism counts are also related: They can
be expressed as linear combinations of each other. Roughly speaking,
embedding counts from $H$ agree with homomorphism counts from $H$
up to ``lower-order terms'' involving only graphs $F$ with strictly
less vertices than $H$. The following lemma follows directly from
\cite[(5.18)]{DBLP:books/daglib/0031021}, and we include a simple
proof (also contained in \cite{DBLP:conf/stoc/CurticapeanDM17}) for
completeness.
\begin{lem}
\label{lem: emb-hom}For any graph $H$, there are coefficients $\beta_{F}\in\mathbb{Z}$
for all graphs $F$ with $|V(F)|<|V(H)|$ such that 
\[
\mathrm{emb}(H\to A)=\hom(H\to A)+\sum_{\substack{\mathrm{graphs}\,F\,\mathrm{with}\\
|V(F)|<|V(H)|
}
}\beta_{F}\hom(F\to A).
\]
\end{lem}

\begin{proof}
Given a partition $\rho$ of the set $V(H)$, the \emph{quotient}
$H/\rho$ is the multigraph obtained by identifying the vertices within
each block of $\rho$ while keeping all possibly emerging self-loops
and multi-edges. We have 
\begin{equation}
\hom(H\to A)=\sum_{\substack{\mathrm{partition}\,\rho\\
\mathrm{of}\,V(H)
}
}\mathrm{emb}(H/\rho\to A)\label{eq: hom-emb-1}
\end{equation}
since any function $f:V(H)\to[n]$ can be viewed equivalently as an
injective function $f:V(H/\rho)\to[n]$ for the partition $\rho=\{f^{-1}(i)\mid i\in[n]\}$.

Write $\bot$ for the finest partition of the set $V(H)$, that is,
the partition consisting of $|V(H)|$ singleton parts. By rearranging
(\ref{eq: hom-emb-1}) and using $H/\bot=H$, we obtain 
\begin{equation}
\mathrm{emb}(H\to A)=\hom(H\to A)-\sum_{\substack{\mathrm{partition}\,\rho\neq\bot\\
\mathrm{of}\,V(H)
}
}\mathrm{emb}(H/\rho\to A).\label{eq: emb-hom-1}
\end{equation}
Note that all graphs $H/\rho$ with $\rho\neq\bot$ have strictly
less vertices than $H$. We can therefore apply (\ref{eq: emb-hom-1})
again to express each term $\mathrm{emb}(H/\rho\to A)$ on the right-hand
side as $\hom(H/\rho\to A)$ minus embedding counts from smaller graphs.
This process can be iterated until reaching single-vertex graphs,
from which homomorphism and embedding counts coincide trivially. Upon
termination, all occurrences of embedding counts are replaced by homomorphism
counts.
\end{proof}
Combining (\ref{eq: cycles-by-format}), \ref{eq: sub-emb-aut}, and
Lemma~\ref{lem: emb-hom}, it follows that the $k$-partial determinant
is a linear combination of homomorphism counts from $k$-partial cycle
covers plus ``lower-order terms''.
\begin{cor}
\label{cor: det-hom}For all $0\leq k\leq n$, there are coefficients
$\alpha_{F}\in\mathbb{Q}$ for all graphs $F$ with $|V(F)|<k$ such
that
\begin{equation}
\mathrm{det}_{k}(A)=\left(\sum_{\lambda\vdash k}\frac{\mathrm{sgn}(C_{\lambda})}{\mathrm{aut}(C_{\lambda})}\hom(C_{\lambda}\to A)\right)+\sum_{\substack{\mathrm{graphs}\,F\,\mathrm{with}\\
|V(F)|<k
}
}\text{\ensuremath{\alpha}}_{F}\hom(F\to A).\label{eq: det-hom}
\end{equation}
\end{cor}

\subsection{Lower-order terms vanish}

It turns out that the ``lower-order terms'' in (\ref{eq: det-hom})
vanish. To show this, we use Kronecker products to lift (\ref{eq: det-hom})
to a polynomial identity and compare coefficients. For $t\in\mathbb{N},$
the \emph{Kronecker product} $A\otimes J_{t}$ of $A$ with the $t\times t$
all-ones matrix $J_{t}$ is an $nt\times nt$ matrix whose rows and
columns are indexed by $[n]\times[t]$, such that the entry at row
$(i,a)$ and column $(j,b)$ equals $a_{i,j}$. In other words, each
entry $a_{i,j}$ of $A$ is ``blown up'' to a $t\times t$ matrix
in $A\otimes J_{t}$, all of whose entries are $a_{i,j}$.

It turns out that both sides of (\ref{eq: det-hom}) react to this
operation ``polynomially'': Fixing $k=n$ and replacing $A$ by
$A\otimes J_{t}$, both sides become polynomials in $t$. To see this,
we consider the homomorphism counts from graphs $F$ occurring on
the right-hand side and observe that $\hom(F\to A\otimes J_{t})$
is proportional to $t^{|V(F)|}$.
\begin{claim}
\label{claim: hom-blowup}For any graph $F$, we have $\hom(F\to A\otimes J_{t})=t^{|V(F)|}\hom(F\to A)$.
\begin{proof}
Every function $f:V(F)\to[n]$ induces $t^{|V(F)|}$ functions $f':V(F)\to[n]\times[t]$
of the same edge-weight product by choosing an index $a_{v}\in[t]$
for each vertex $v\in V(F)$. Conversely, every function $g:V(F)\to[n]\times[t]$
is induced by the function that forgets the second component of the
images.
\end{proof}
\end{claim}

It follows from (\ref{eq: det-hom}) and Claim~\ref{claim: hom-blowup}
that
\begin{equation}
\mathrm{det}_{n}(A\otimes J_{t})=t^{n}\left(\sum_{\lambda\vdash n}\frac{\mathrm{sgn}(C_{\lambda})}{\mathrm{aut}(C_{\lambda})}\hom(C_{\lambda}\to A)\right)+\sum_{\substack{\mathrm{graphs}\,F\,\mathrm{with}\\
|V(F)|<n
}
}t^{|V(F)|}\alpha_{F}\hom(F\to A).\label{eq: lifted-det-hom}
\end{equation}
An elementary property of the determinant implies that $\mathrm{det}_{n}(A\otimes J_{t})$
is proportional to $t^{n}$, so the sum over lower-order graphs $F$
in (\ref{eq: lifted-det-hom}) vanishes.
\begin{claim}
\label{claim: det-blowup}We have $\mathrm{det}_{n}(A\otimes J_{t})=t^{n}\mathrm{det}_{n}(A)$.
\begin{proof}
By definition of the partial determinant, we have
\[
\mathrm{det}_{n}(A\otimes J_{t})=\sum_{\substack{S\subseteq[n]\times[t]\\
\mathrm{with}\:|S|=n
}
}\det((A\otimes J_{t})[S]).
\]
If $S$ contains two pairs that agree in the first component, then
the $n\times n$ matrix $(A\otimes J_{t})[S]$ contains two equal
rows (and two equal columns), so its $n$-determinant vanishes. We
can therefore restrict the summation to index sets of the form $S=\{(1,a_{1}),\ldots,(n,a_{n})\}$
for $a_{1},\ldots,a_{n}\in[t]$. There are $t^{n}$ sets $S$ of this
form, each with $(A\otimes J_{t})[S]=A$. The claim follows.
\end{proof}
By comparing coefficients in the polynomial identity (\ref{eq: lifted-det-hom})
and using from Claim~\ref{claim: det-blowup} that the left-hand
side only contains the monomial $t^{n}$, it follows that
\[
\mathrm{det}_{n}(A\otimes J_{t})=t^{n}\left(\sum_{\lambda\vdash n}\frac{\mathrm{sgn}(C_{\lambda})}{\mathrm{aut}(C_{\lambda})}\hom(C_{\lambda}\to A)\right).
\]
Setting $t=1$, we obtain
\begin{align}
\mathrm{det}(A)=\mathrm{det}_{n}(A) & =\sum_{\lambda\vdash n}\frac{\mathrm{sgn}(C_{\lambda})}{\mathrm{aut}(C_{\lambda})}\cdot\hom(C_{\lambda}\to A)\label{eq: det-to-hom}\\
 & =\sum_{\lambda\vdash n}(-1)^{n+|\lambda|}\prod_{\ell=1}^{n}\frac{\mathrm{tr}(A^{\ell})^{s_{\ell}(\lambda)}}{s_{\ell}(\lambda)!\cdot\ell^{s_{\ell}(\lambda)}}.
\end{align}
In the last equation, $\hom(C_{\lambda}\to A)$, $\mathrm{aut}(C_{\lambda})$,
and $\sgn(C_{\lambda})$ are expanded via (\ref{eq: hom-cc}), (\ref{eq: aut-cc}),
and (\ref{eq: def-partial-sign}), respectively. This proves (\ref{eq: det-partsum}).
Note that, up to this last equation, the only properties about determinants
used are that 
\end{claim}

\begin{enumerate}
\item $k$-partial determinants can be expressed as a linear combination
of $k$-vertex subgraph counts, and that 
\item the determinant vanishes on a matrix with two equal rows or columns. 
\end{enumerate}
Thus, the proof applies to any matrix functional satisfying these
two properties.

\section{\label{sec:Algorithmic-applications}Algorithmic applications}

Equation~(\ref{eq: det-partsum}) does not directly imply a polynomial-time
algorithm for the determinant, as the sum over partitions $\lambda\vdash n$
involves $\exp(O(\sqrt{n}))$ terms. Nevertheless, this sum can be
computed in polynomial time, e.g.~via dynamic programming or polynomial
multiplication, as shown below.
\begin{lem}
\label{lem: dp}Given all quantities $\mathrm{tr}(A^{\ell})$ for
$1\leq\ell\leq n$, we can compute $\mathrm{det}(A)$ with $O(n^{3})$
operations.
\end{lem}

\begin{proof}
Let $X$ be a formal indeterminate. For $1\leq\ell\leq n$, we define
the polynomial
\[
p_{\ell}(X)=\sum_{i=0}^{\lfloor n/\ell\rfloor}(-1)^{i}\frac{\mathrm{tr}(A^{\ell})^{i}}{i!\cdot\ell^{i}}X^{\ell i}.
\]
Then we compute the first $n+1$ coefficients of the product $p_{1}\cdot\ldots\cdot p_{n}$.
Up to the factor $(-1)^{n}$, the coefficient of $X^{n}$ in this
product is the right-hand side of (\ref{eq: det-partsum}): The coefficient
is a weighted count over all ways to choose a monomial $X^{i_{\ell}}$
from each polynomial $p_{\ell}$, subject to $\sum_{\ell}\ell\cdot i_{\ell}=n$,
thus encoding a partition $(1^{i_{1}},\ldots n^{i_{n}})\vdash n$.
In this count, each partition is weighted by $\prod_{\ell=1}^{n}(-1)^{i_{\ell}}\frac{\mathrm{tr}(A^{\ell})^{i_{\ell}}}{i_{\ell}!\cdot\ell^{i_{\ell}}}$.
The coefficient of $X^{n}$ in $p_{1}\cdot\ldots\cdot p_{n}$ thus
equals the right-hand side of (\ref{eq: det-partsum}).

The first $n+1$ coefficients of $p_{1}\cdot\ldots\cdot p_{n}$ can
be computed in $O(n^{3})$ time by iteratively multiplying $p_{t}$
onto $p_{1}\cdot\ldots\cdot p_{t-1}$ and truncating the intermediate
result to the first $n+1$ coefficients. Using standard polynomial
multiplication, each of the $n$ iterations takes $O(n^{2})$ operations.
\end{proof}
We can compute $\mathrm{tr}(A^{\ell})$ for all $1\leq\ell\leq n$
in $O(n^{\omega+1})$ overall time. This implies:
\begin{thm}
The determinant $\mathrm{det}(A)$ can be computed with $O(n^{\omega+1})$
operations.
\end{thm}

The traces can also be computed with arithmetic constant fan-in circuits
of poly-logarithmic depth: Any individual matrix product can trivially
be computed in $O(\log n)$ depth and $O(n^{3})$ size. Repeated squaring
allows us to compute all matrix powers $A^{\ell}$ and their traces
for $1\leq\ell\leq n$ in $O(\log^{2}n)$ depth and $\tilde{O}(n^{4})$
overall size. After all traces are computed, the polynomial multiplications
from the proof of Lemma~\ref{lem: dp} can be performed in $O(\log^{2}n)$
depth and $\tilde{O}(n^{3})$ size: A single polynomial multiplication
requires $O(\log n)$ depth and $\tilde{O}(n^{2})$ size. The truncated
product of the $n$ relevant polynomials can then be computed as an
$O(\log n)$-depth binary tree in which each node consists of a polynomial
multiplication followed by truncating to the first $n+1$ coefficients.
This implies:
\begin{thm}
The determinant $\mathrm{det}(A)$ can be computed with an arithmetic
bounded fan-in circuit of $O(\log^{2}n)$ depth and $\tilde{O}(n^{4})$
size.
\end{thm}

\bibliographystyle{plain}
\bibliography{simple-det}

\end{document}